\documentclass[12pt]{iopart}
\usepackage{graphicx}

\expandafter\let\csname equation*\endcsname\relax
\expandafter\let\csname endequation*\endcsname\relax

\usepackage{amsthm}
\usepackage{amsmath}
\usepackage{iopams}
\usepackage{amscd,amsmath,amsopn,amssymb,amsthm,multicol,mathtools}
\usepackage{graphicx}
\usepackage{amssymb} 

\newcounter{oftheorem}[section]
\newenvironment{mytheorem}[1]%
{\begin{trivlist}
     \renewcommand{\theoftheorem}{#1~\thesection.\arabic{oftheorem}}
     \refstepcounter{oftheorem}
     \item[\hspace{\labelsep}\bf\thesection.\arabic{oftheorem} #1.]}%
{\end{trivlist}}
\newenvironment{definition}{\begin{mytheorem}{Definition}}{\end{mytheorem}}
\newenvironment{example}{\begin{mytheorem}{Example}\it}{\end{mytheorem}}
\newenvironment{proposition}{\begin{mytheorem}{Proposition}\it}{\end{mytheorem}}
\newenvironment{theorem}{\begin{mytheorem}{Theorem}\it}{\end{mytheorem}}

\newenvironment{remark}{\begin{mytheorem}{Remark}}{\end{mytheorem}}
\newenvironment{lemma}{\begin{mytheorem}{Lemma}}{\end{mytheorem}}

\usepackage{hyperref}

\begin{document}

\title{Dynamical systems on the Liouville plane and the related strictly contact systems}

\author{Stavros Anastassiou}
\address{Center of Research and Applications of Nonlinear Systems (CRANS)\\ 
 University of Patras, Department of Mathematics\\GR-26500 Rion, Greece\\SAnastassiou@gmail.com}

\begin{abstract}
We study vector fields of the plane preserving the form of Liouville. 
We present their local models up to the natural equivalence relation, 
and describe local bifurcations of low codimension. To achieve that, 
a classification of univariate functions is given, according to a 
relation stricter than contact equivalence. We discuss, in addition, 
their relation with strictly contact vector fields in dimension three. Analogous results 
for diffeomorphisms are also given.
\end{abstract}

\vspace*{1.cm}

\maketitle
\section*{Introduction}
Dynamical systems preserving a geometrical structure have been studied quite 
extensively. Especially those systems preserving a symplectic form have attracted 
a lot of attention, due to their fundamental importance in all kinds of applications. 
Dynamical systems preserving a contact form are also of interest, both in mathematical 
considerations (for example, in classifying partial differential equations) and in 
specific applications (study of Euler equations).

The 1--form of Liouville may be associated both with a symplectic form (by taking the exterior 
derivative of it) and with a contact form (by adding to it a simple 1--form of a new variable). 
We wish here to study dynamical systems respecting the form of Liouville. As we shall see, they 
are symplectic systems which may be extented to contact ones.

To set up the notation, 
let M be a smooth (which, in this work, means continuously differentiable the sufficient number of times) 
manifold of dimension $2n+1$. A contact 
form on M is a 1-form $\alpha $ such that $\alpha \wedge (d\alpha)^n \neq 0$. 
A strict contactomorphism is a diffeomorphism of M which preserves the contact form 
(their group will be denoted as $Diff(M,a)$) while a vector field on M is called strictly 
contact if its flow consists of strict contactomorphims (we denote their algebra as $\mathcal{X}(M,a)$). 
In terms of the defining contact form $\alpha $, we have 
$f^*\alpha =\alpha$ for a strict contactomorphism f and $\mathcal{L}_X\alpha =0$ 
for a strictly contact vector field X, where $\mathcal{L}_X\alpha$ denotes the Lie derivative of 
$\alpha$ in the direction of the field X. The classical example of a strictly contact vector field 
associated to $\alpha$ is the vector field of Reeb, $\mathcal{R}_a$, uniquely defined by 
the equations $\alpha (\mathcal{R}_\alpha)=1$ and $d\alpha(\mathcal{R}_\alpha,\cdot )=0$.

Associated to every contact vector field X is a smooth function $H:M\rightarrow \mathbb{R}$, 
called the contact Hamiltonian of X, which is given as $H=\alpha(X)$. Conversely, 
every smooth function $H$ gives rise to a unique contact vector field $X$, such that 
$a(X)=H$ and $d\alpha(X,\cdot )=(\mathcal{L}_{\mathcal{R}_\alpha}H)\alpha(\cdot)-dH(\cdot)$. 
Usually we write $X_H$ to denote the dependence of vector field $X_H$ on its (contact) 
Hamiltonian function $H$.

Results conserning the local behavior for systems of this kind may be found in 
\cite{chaperon, chaperon2, llave, lychagin}, where the authors provide explicit conditions 
for their linearization, in the neighborhood of a hyperbolic singularity. The study of 
degenerate zeros, and of their bifurcations, remains, however, far from complete.
 
Here, in section 1, we recall the form of strictly contact vector fields of $\mathbb{R}^3$, 
and their relation with symplectic vector fields of the plane. We show that the 
albegra $\mathcal{X}(\mathbb{R}^2,xdy)$ of plane fields preserving the form of Liouville $xdy$ 
may be obtained by projecting on $\mathbb{R}^2$ stictly contact fields with constant third component. 
We begin the classification of vector fields belonging in $\mathcal{X}(\mathbb{R}^2,xdy)$ 
(we shall call them Liouville vector fields) by introducing the natural equivalence relation, and 
by showing that the problem of their classification is equivalent to a classification of functions 
up to a specific equivalence relation.

In section 2, (germs at the orign of) univariate functions are classified up to this equivalence 
relation, which we name ``restricted contact equivalence'', due to its similarity with 
the classical contact equivalence of functions. We provide a complete list of normal forms 
for function germs up to arbitrary (finite) codimension.

In section 3, based on the previous results, we give local models for Liouville vector fields 
of the plane. We first prove that all such fields are conjugate at points where they do not vanish, 
then we prove that they can be linearized at hyperbolic singularities, and finally we 
state the result conserning their finite determinacy, which is based on the finite determinacy 
theorem obtaind in section 2.

In section 4, we first show how to construct a transversal unfolding of a singularity class 
of Liouville vector fields and then we present transversal unfoldings for singularity classes 
of codimension 1 and 2. Phase portraits for generic bifurcations of members of 
$\mathcal{X}(\mathbb{R}^2,xdy)$ are also given.

Next, in section 5, we see that there is only one polynomial member of the group of plane 
diffeomorphisms preserving the form of Liouville ($Diff(\mathbb{R}^2,xdy)$ stands for this group). 
This is the linear Liouville diffeomorphism, and we show the linearization of plane diffeomorphisms 
of this kind at hyperbolic fixed points. 

In section 6, we return to members of $\mathcal{X}(\mathbb{R}^3,a)$ to observe that 
the models obtained above are members of a specific base of the vector space of 
homogeneous vector fields. Their linearization is again shown, albeit using classical 
methods of normal form theory.

Last section contains some observations concerning future directions.

For a classical introduction to symplectic and contact topology the reader should consult 
\cite{arnold2}, while \cite{geiges} offers a more complete study of the contact 
case. Singularities of mappings are treated in a number of textbooks; we recommend \cite{arnold1,germs} 
and \cite{damon} (see \cite{ninakos} for a recent application of singularity theory to problems 
of dynamics). 
\section{Strictly contact vector fields and fields of Liouville}
Let M be a closed smooth manifold of dimension 2n+1 equipped with a contact form $\alpha$. 
The  contact form is called regular if its Reeb vector field, $\mathcal{R}_{\alpha}$, 
generates a free $\mathbb{S}^1$ action on M. In this case, M is the total space of 
a principal $\mathbb{S}^1$ bundle, the so called Boothby-Wang bundle 
(see \cite{boothby} for more details):
\begin{center}
$\mathbb{S}^1\xrightarrow{k} M\xrightarrow{\pi} B$,
\end{center}
where $k:\mathbb{S}^1\rightarrow M$ is the action of the Reeb field and $\pi:M\rightarrow B$ 
is the canonical projection on $B=M/\mathbb{S}^1$. B is a symplectic manifold with 
symplectic form $\omega =\pi _* da$. The projection $\pi$ induces an algebra isomorphism 
between functions on the base B and functions on M which are preserved under the flow of 
$\mathcal{R}_{\alpha}$ (such functions are called basic). It also induces a surjective 
homomorphism between strictly contact vector fields X of $(M,\alpha)$ and hamiltonian 
vector fields Y of $(B,\omega)$ 
(that is, fields Y with $\mathcal{L}_Y\omega=0$), the kernel of which homomorphism is 
generated by the vector field of Reeb.

In our local, three dimensional, case, things are of course simpler. Using a local 
Darboux chart, consider the euclidean space $\mathbb{R}^3$ equipped with the standard 
contact structure $\alpha=dz+xdy$. Its Reeb vector fiel, $\mathcal{R}_{\alpha}=\frac{\partial}{\partial z}$, 
induces the action $\varphi ^t(x,y,z)=(x,y,z+t)$, and the quotient of $\mathbb{R}^3$ 
by this action, that is, the plane $\mathbb{R}^2$ with coordinates $(x,y)$, inherits 
the symplectic form $\omega=\pi_*d\alpha=dx\wedge dy$. Strictly contact vector 
fields of $\mathbb{R}^3$ project to hamiltonian fields on this plane (for a direct analogy 
with the volume--preserving case the reader should consult \cite{haller}).

Basic functions now depend, as one may easily verify, only on the first two variables, while the kernel of 
the above mentioned projection contains the multiples of $\frac{\partial}{\partial z}$. 
Studying equation $\mathcal{L}_X\alpha=0$ we get the general expression of $X=X_1\frac{\partial}{\partial x}+X_2\frac{\partial}{\partial y}+X_3\frac{\partial}{\partial z}\in \mathcal{X}(\mathbb{R}^3,\alpha )$:
\begin{center} 
$X=(-\frac{\partial}{\partial y}\int X_2(x,y)dx)\frac{\partial}{\partial x}+X_2(x,y)\frac{\partial}{\partial y}+(-xX_2(x,y)+\int X_2(x,y)dx)\frac{\partial}{\partial z}.$
\end{center} 
Its contact Hamiltonian is of course $H(x,y,z)=\int X_2(x,y)dx$ (recall that it does not depend on the 
third variable), thus:
\begin{center}
$X=-\frac{\partial H(x,y,z)}{\partial y}\frac{\partial}{\partial x}+\frac{\partial H(x,y,z)}{\partial x}\frac{\partial}{\partial y}+(H(x,y,z)-x\frac{\partial H(x,y,z)}{\partial x})\frac{\partial}{\partial z}.$
\end{center}
Observe that all vector fields of the $(x,y)$--plane, preserving the symplectic structure 
$dx\wedge dy$, may be obtained in this way.

In this work we restrict our attention to those members of $\mathcal{X}(\mathbb{R}^3,\alpha )$, 
which preserve the form of Liouville $xdy$ (we shall denote their set as $\mathcal{X}_L(\mathbb{R}^3,\alpha )$). 
The reason for this choise will become clear in section 6.
In this case, equation $\mathcal{L}_X \alpha =0$ becomes:   
\begin{equation}
\label{basiceq}
X_1(x,y)=-x\frac{d X_2(y)}{dy},
\end{equation}
while  $X_3(x,y)=c\in\mathbb{R}$. Thus, their general form is 
$-x\frac{dh(y)}{dy}\frac{\partial}{\partial x}+h(y)\frac{\partial}{\partial y}+c\frac{\partial }{\partial z}$, 
for some univariate function $h(y)$ and a constant $c$. Observe that all vector fields 
of the plane presrving the form of Liouville may be obtained by projecting the members of 
$\mathcal{X}_L(\mathbb{R}^3,\alpha )$ on the $z=0$ plane. We have, therefore, the following:
\begin{lemma}
\label{basiclemma}
To every $h\in\mathcal{C}^k(\mathbb{R},\mathbb{R}),\ k\geq 2$, 
corresponds a unique $X\in\mathcal{X}(\mathbb{R}^2,xdy)$, namely 
$-x\frac{dh(y)}{dy}\frac{\partial}{\partial x}+h(y)\frac{\partial}{\partial y}$. Members of 
$\mathcal{X}_L(\mathbb{R}^3,\alpha )$ are trivially obtained by adding constant multiples of 
$\frac{\partial}{\partial z}$ to members of $\mathcal{X}(\mathbb{R}^2,xdy)$.
\end{lemma}
This lemma provides the general form of the vector fields we are interested in. 

Our goal is the classification of these vector fields according 
to the natural relation defined in the obvious 
way: two fields $X,\ Y\in \mathcal{X}(\mathbb{R}^2,xdy)$ are 
Liouville conjugate if there exists a diffeomorphism of the plane preserving the form of Liouville, 
$\phi \in Diff(\mathbb{R}^2,xdy)$, such that $\phi_*X=Y$, while two fields 
$Z,\ W \in \mathcal{X}(\mathbb{R}^3,a)$ are strictly contact conjugate if a $\psi \in Diff(\mathbb{R}^3,\alpha )$ 
exists, such that $\psi _* Z=W$. Observe that classifying members of $\mathcal{X}(\mathbb{R}^2,xdy)$ 
leads to a classification of fields belonging in $\mathcal{X}_L(\mathbb{R}^3,a)$; one 
needs only to extend $\phi$ to $\mathbb{R}^3$ as $\psi(x,y,z)=(\phi(x,y),z)$.

To proceed with the classification of Liouville vector fields of the plane, 
we shall exploit their dependence on real valued functions.

\begin{lemma}
\label{transforming}
Let $f$ be a univariate function and $\varphi$ a diffeomorphism of $\mathbb{R}$. 
The Liouville vector field corresponding to function $f$ may be transformed, 
via a diffeomorphism respecting the form $xdy$, to the Liouville vector field 
corresponding to the function $\frac{1}{\phi '(y)}f(\phi(y))$. 
\end{lemma}
\begin{proof}
Constructing the fields corresponding to these two functions, according to the recipe given 
in lemma \ref{basiclemma}, we conclude that the diffeomorphism accomplishing the desired 
transformation is $\psi (x,y)=(\frac{x}{\phi '(y)},\phi (y))$ which also preserves the Liouville form.
\end{proof}
This lemma ensures that the classification of Liouville vector fields, up to 
diffeomorphisms belonging in $Diff(\mathbb{R}^2,xdy)$, reduces to a classification 
of univariate real functions. In the next section, we turn our attention to this classification.
\section{Restricted contact equivalence}
Let $f:(\mathbb{R},0)\rightarrow (\mathbb{R},0)$ be the germ at the origin of 
a smooth function. Their ring will be denoted as $\mathcal{E}$. We introduce the 
following equivalence relation.
\begin{definition}
 Let $f,g\in \mathcal{E}$. We shall call them restrictively contact equivalent 
 ($\mathcal{RK}$-equivalent) 
 if there exists a germ of a smooth diffeomorphism $\varphi:(\mathbb{R},0)\rightarrow (\mathbb{R},0)$ such 
 that $g=\frac{1}{\varphi '}(f\circ \varphi)$.
\end{definition}
\begin{example}
 Let $f,g\in \mathcal{E}$, with $f(x)=x,\ g(x)=x+x^2$. Define $\varphi (x)=\frac{x}{x+1}$. It is easy 
 to check that $\varphi$ is a local diffeomorphism at $0\in \mathbb{R}$ and $g=\frac{1}{\varphi '}(f\circ \varphi)$.
\end{example}
Let us recall here that two univariate function germs $f,g\in \mathcal{E}$ are called contact equivalent 
if $f(x)=M(x)g(\varphi(x))$, for some function germ $M(x)$ and diffeomorphism $\varphi$. The 
equivalence relation we study here requires $M(x)=\frac{1}{\varphi '(x)}$. This explains 
why we called the above defined equivalence relation restricted contact.

Suppose now that $g_s\in \mathcal{E}$ is a curve of $\mathcal{RK}$--equivalent germs, depending 
on the real parameter $s$, with $g_0=f$. There exists thus a curve of local diffeomorphisms 
$\varphi _s:(\mathbb{R},0)\rightarrow (\mathbb{R},0)$, with $\varphi_s(0)=0, \forall s\in \mathbb{R}$ 
and $\varphi_0(x)=x$, such that $g_s(x)=\frac{1}{\varphi _s'(x)}f(\varphi_s (x))$. 
Differenting with respect to s and evaluating at $s=0$ we get:
\begin{center}
 $\frac{\partial}{\partial s}g_s(x)|_{s=0}=-X'(x)f(x)+f'(x)X(x)$,
\end{center}
where $X(x)$ is defined by the relation $\frac{\partial}{\partial s}\varphi_s (x)=X(\varphi_s(x))$. Note that $X(0)=0$, 
thus $X(x)\in m$, the ideal of $\mathcal{E}$ generated by $x\in \mathcal{E}$.
\begin{lemma}
 Let $f\in \mathcal{E}$. The ideal generated from the germs $-X'(x)f(x)+f'(x)X(x),\ X\in m$,  
 equals $\langle f(x)\rangle+f'(x)m$.
\end{lemma}
\begin{proof}
 It is obvious that, if $X(x)\in m$, then $-X'(x)f(x)+f'(x)X(x)$ is a member of 
 $\langle f(x)\rangle+f'(x)m$. Let us prove the opposite inclusion.
 
 Let $h\in \langle f(x)\rangle+f'(x)m$. Germs $g\in \mathcal{E}$ and $k\in m$ exist, such that 
 $h(x)=g(x)f(x)+f'(x)k(x)$. We wish to find a germ $X\in m$ such that:
 \begin{center}
  $h(x)=-X'(x)f(x)+X(x)f'(x)\Rightarrow g(x)f(x)+f'(x)k(x)=-X'(x)f(x)+X(x)f'(x)$.
 \end{center}
 One may easily check that a solution of the last differential equation is:
 \[   
X(x) = 
     \begin{cases}
       k(x)-f(x)\int _0^x\frac{g(t)+k'(t)}{f(t)}dt &\quad \text{if } x\neq 0\\
       0 &\quad\text{if } x=0 \ 
     \end{cases}
\]
 which is well defined and smooth in a neighborhood of the origin and, therefore, 
 for every $h\in \langle f(x)\rangle+f'(x)m$ a $X\in m$ 
 exists, such that $h=-X'(x)f(x)+X(x)f'(x)$, hence the conclusion.
\end{proof}
Under the light of the lemma above, we proceed to the following:
\begin{definition}
 The tangent space of $f\in \mathcal{E}$, with respect to $\mathcal{RK}$--equivalence, is defined to 
 be $T_{RK}f:=\langle f(x)\rangle+f'(x)m$. The codimension of $f$ is defined as $codim_{RK}(f):=dim(m/T_{RK}f).$ 
\end{definition}
\begin{example}
 We calculate that, if $f(x)=x$, then $T_{RK}f=m$, thus $codim_{RK}(f)=0$, while 
 if $g(x)=x^2$, $T_{RK}g=m^2$ and $codim_{RK}(g)=1$.
\end{example}
As usual, the germ $f\in \mathcal{E}$ is called $k$--determined, with $k\in \mathbb{N}$, 
if every other $g\in \mathcal{E}$ having the same $k$--jet with $f$ is $\mathcal{RK}$--equivalent to f. 
If such a finite $k$ does not exist, we say that $f$ is not finitely determined.
\begin{theorem}
 The germ $f\in \mathcal{E}$ is $k$--determined, with respect to $\mathcal{RK}$-equivalence, if 
 $m^{k+1}\subseteq mT_{RK}f$. 
\end{theorem}
\begin{proof}
 We have to prove that if $h\in m^{k+1}\subseteq mT_{RK}f$, the germs $f$ and $f+h$ are 
 $\mathcal{RK}$--equivalent.
 
 Towards this end, define $f_s=f+sh,\ s\in[0,1]$. We shall construct diffeomorphisms $\varphi_s(x)$, defined in a 
 neighborhood of the origin, such that $f_s=\frac{1}{\varphi '_s}f(\varphi_s(x))$.  Differentiating with respect 
 to $s$, we get:
 \begin{center}
  $h(x)=-\frac{1}{\varphi '_s(x)}X'(\varphi _s(x))f(\varphi _s(x))+\frac{1}{\varphi '_s(x)}X(\varphi _s(x))f'(\varphi _s(x))$.
 \end{center}
Note that, for $s=0$, we get the relation $h(x)=-X'(x)f(x)+X(x)f'(x)$, which, by the previous lemma, 
has a solution $X(x)\in m$ since $m^{k+1}\subseteq mT_{RK}f$. We need to show that a solution exists 
for all $s\in [0,1]$.

Consider $\mathbb{R}\times [0,1]$, let $\mathcal{R}$ be the ring of function germs at ${0}\times [0,1]$ 
and denote by $m_s$ the ideal of $\mathcal{R}$ consisting of those germs vanishing at ${0}\times [0,1]$.
We have:
\begin{center}
$ m^{k+1} \subseteq m^{k+1}_s$\\
$\ \ \ \ \ \ \ \ \ \ \ \ \ \ \ \ \ \ \ \ \ \ \ \ \ \subseteq m_s\langle f \rangle_{\mathcal{R}}+f'(x)m^2_s$\\
$\ \ \ \ \ \ \ \ \ \ \ \ \ \ \ \ \ \ \ \ \ \ \ \ \ \ \ \ \ \ \ \ \ \ \ \ \ \ \ \ \ \ \ \ \ \ \subseteq m_s\langle f_s\rangle _{\mathcal{R}}+m_s\langle h\rangle _{\mathcal{R}}+f'_sm^2_s+h'm^2_s$\\
$\ \ \ \ \ \ \ \ \ \ \ \ \ \ \ \ \ \ \ \ \ \ \ \ \ \ \ \ \ \ \ \ \ \subseteq m_s \langle f_s\rangle _{\mathcal{R}}+f'_sm^2_s+m_s^{k+2}$\\
$\ \ \ \ \ \ \ \ \ \ \ \ \ \ \ \ \ \ \ \ \ \ \ \ \ \ \ \ \ \ \ \ \ \ \ \ \ \subseteq m_s \langle f_s\rangle _{\mathcal{R}}+f'_sm^2_s+m_s m_s^{k+1}$\\
$\ \ \ \ \ \ \ \ \ \ \ \ \ \ \ \ \ \ \ \ \ \ \ \ \ \ \subseteq m_s(\langle f_s \rangle _{\mathcal{R}}+f'_sm_s)$,
\end{center}
where the inclusion in the last line holds due to the Nakayama lemma. Thus, for every $s\in [0,1]$, 
we have that $h\in m^{k+1} \subseteq m_s T_{RK}f_s$. We can therefore find $X_s(x)\in m_s$, 
defining the germ of diffeomorphism $\varphi _s$ which, for $s=1$, establishes an equivalence between 
$f$ and $f+h$.
\end{proof}
The classification of the elements of $\mathcal{E}$ now follows. We begin with germs that 
either do not vanish at the origin, or have a regular point there.
\begin{lemma}
 Let $f\in \mathcal{E}$. If $f(0)\neq 0$, it is $\mathcal{RK}$--equivalent to 1, while 
 if $f(0)=0$ and $f'(0)=a\neq 0$, $f$ is $\mathcal{RK}-equivalent$ to $ax$.
\end{lemma}
\begin{proof}
 Let $f\in \mathcal{E}$, with $f(0)\neq 0$. To show that it is $\mathcal{RK}$--equivalent 
 to $1$, we must find a local diffeomorphism $k(x)$ such that $\frac{1}{k'(x)}=f(x)$, which is 
 the same as $k'(x)=\frac{1}{f(x)}$, which is a differential equation with smooth right hand side, 
 at least in a neighborhood of the origin, thus, such a smooth $k(x)$ exists.
 
 On the other hand, let $f(0)=0$ and $f'(0)=a\neq 0$. It is 1--determined, thus $\mathcal{RK}$--equivalent 
 to its linear part $ax$, while, as may be easily verified, the germ $ax$ is $\mathcal{RK}$--equivalent 
 to $bx$ only if $a=b$.
 \end{proof}
Let us know proceed to germs with critical points.
\begin{lemma}
Let $f\in \mathcal{E}$, with $f(0)=f'(0)=...=f^{k-1}(0)=0$ and $f^k(0)\neq 0,\ k>1$. Then 
$f$ is $\mathcal{RK}$--equivalent to $x^k$, if $k$ is an even number and to $x^{k}$ or 
$-x^{k}$, if $k$ 
is an odd number.
\end{lemma}
\begin{proof} 
If $f$ is such a germ, then, in a neighborhood of the origin, we may write $f(x)=x^kg(x)$, with 
$g(0)\neq 0$. Thus $T_{RK}f=m^k$, and $f$ is $k$--determined. It is thus $\mathcal{RK}$--equivalent 
to $ax^k$, while, as may easily be verified, the germ of a diffeomorphism $\varphi (x)$ exists such 
that $\frac{1}{\varphi '(x)}a\varphi ^k(x)=x^k$, for every $a\in \mathbb{R} \setminus \{0\}$, if $k$ is 
even, while if $k$ is odd then $ax^k$ is $\mathcal{RK}$--equivalent to $-x^k$, for $a<0$ and to $x^k$, for $a>0$.
\end{proof}
Combining all the above, we may now state the main theorem for the classification of members of $\mathcal{E}$.
\begin{theorem}
\label{functions}
 If a member of $\mathcal{E}$ does not vanish at the origin it is $\mathcal{RK}$-equivalent to the constant function 
 $1$. Members of $\mathcal{E}$ having codimension $0$ are $\mathcal{RK}$-equivalent to $ax$ ($a$ being 
 the value of their derivative there). A member of $\mathcal{E}$ of odd codimension $k$ 
 is $\mathcal{RK}$-equivalent to $x^{k+1}$, while if it is of even codimension $k$ it is 
 $\mathcal{RK}$-equivalent to $\pm x^{k+1}$, depending on the sign of the value of its first non--vanishing 
 derivative at the orgin.
\end{theorem}

Table 1 contains the local models of members of $\mathcal{E}$ having codimension up to five. We note that 
there are differences with the classical classification list for right equivalence (in which list the $A_1,\ A_3$ and 
$A_5$ models may have both negative and positive sign) and for contact equivalence (in which, for example, the 
$A_0$ model does not depend on the constant $a$, see \cite{arnold1,damon}). The interested reader 
should consult \cite{takahashi} for a relation of contact and right equivalence, while the equivalence 
relation studied here provides more models than right and contact equivalence since 
it is stricter than both.
\begin{center}
Table 1
\end{center}    
{\small \begin{center}
   \begin{tabular}{lll}
   \hline
   symbol & codimension & function \\
\hline 
$\ $  & $\ $ & $1$ \\ 
  $A_0^{a}$  & $0$ & $ay$\\
  $A_1$  & $1$ & $y^2$ \\
   $A_2^{\pm}$ & $2$ & $\pm y^3$ \\
   $A_3$ & $3$ & $y^4$ \\
   $A_4^{\pm}$ & $4$ & $\pm y^5$ \\
   $A_5$ & $5$ & $y^6$ 
       \end{tabular}
        \end{center}}

\section{Local models for members of $\mathcal{X}(\mathbb{R}^2,xdy)$}
We return now to our study of vector fields of the plane, which preserve 
the form of Liouville. To construct their local models, we make use of lemma 
\ref{transforming} along with theorem \ref{functions}.

\begin{lemma}(of regular points)
Let $X\in \mathcal{X}(\mathbb{R}^2,xdy)$ be such that $X(0)\neq 0$. Then, in 
a neighborhood of zero, it is conjugate, via a diffeomorphism preserving the form of 
Liouville, to the constant vector field $\frac{\partial}{\partial y}$.
\end{lemma}
\begin{proof}
Since $X\in \mathcal{X}(\mathbb{R}^2,xdy)$, it is of the form $X(x,y)=-xf'(y)\frac{\partial}{\partial x}+f(y)\frac{\partial}{\partial y}$, 
for a smooth, real valued, function f(y). Since $X(0,0)=f(0)\frac{\partial}{\partial y}\neq 0$, we get $f(0)\neq 0$, 
which means that $f$ is $\mathcal{RK}$--equivalent to the constant function 1, thus, by lemma \ref{transforming}, 
a diffeomorphism preserving $xdy$ exists, transforming $X$ to $\frac{\partial}{\partial y}$.
\end{proof}
Let us now turn our attention to hyperbolic singularities.
\begin{lemma}(hyperbolic singularities)
Let $X\in \mathcal{X}(\mathbb{R}^2,xdy)$ having a hyperbolic singularity at the origin. 
Then, in a neighborhood of zero, it is conjugate, via a diffeomorphism preserving the form of 
Liouville, to the vector field $-ax\frac{\partial}{\partial x}+ay\frac{\partial}{\partial y}$.
\end{lemma}
\begin{proof}
 The vector field is of the form $-xf'(y)\frac{\partial}{\partial x}+f(y)\frac{\partial}{\partial y}$, 
 and it is easy to check that the eigenvalues of zero are $-f'(0)$ and $f'(0)$. Thus 
 zero is a hyperbolic singularity if, and only if, $f'(0)\neq 0$, and $f$ is therefore $\mathcal{RK}$--equivalent 
 to $ay$, $a=f'(0)$. The existence of a diffeomorphism transforming $X$ to 
 $-ax\frac{\partial}{\partial x}+ay\frac{\partial}{\partial y}$ is guarantied, by lemma \ref{transforming}.
\end{proof}

We see that, at a hyperbolic singularity, all members of $\mathcal{X}(\mathbb{R}^2,xdy)$ are 
topologically equivalent: they are of the saddle type. Up to diffeomorphisms respecting the 
form of Liouville, however, their equivalence classes are classified by a real number.

The lemmata above ensure that the first non--vanishing jet of members of $\mathcal{X}(\mathbb{R}^2,xdy)$ 
completely determine their local behavior, at least in the simplest cases. Actually, this holds in 
general.
\begin{theorem}
Let $X,Y\in \mathcal{X}(\mathbb{R}^2,xdy)$. If $j^kX(0)=j^kY(0)=0,k=0,..,i-1,$ and 
$j^iX(0)=j^iY(0)\neq 0,$ for some $i\in \mathbb{N}\setminus \{0\}$, there exists a diffeomorphism 
preserving $xdy$ which conjugates $X$ and $Y$.
\end{theorem}

We ommit the proof, since it follows the lines of the lemma classifying the hyperbolic singularities. 
Using theorem \ref{functions}, we give in Table 2 the local models of singularities of members of 
$\mathcal{X}(\mathbb{R}^2,xdy)$, up to codimension 5.

\begin{center}
Table 2
\end{center}    
{\small \begin{center}
   \begin{tabular}{lll}
   \hline
   symbol & codimension & local model \\
\hline 
$\ $  &  $\ $ & $\frac{\partial}{\partial y}$  \\ 
  $A_0^{a}$  &  $0$  & $-ax\frac{\partial}{\partial x}+ay\frac{\partial}{\partial y}$ \\
  $A_1$  & $1$ & $-2xy\frac{\partial}{\partial x}+y^2\frac{\partial}{\partial y}$  \\
   $A_2$  & $2$ & $-3xy^2\frac{\partial}{\partial x}+y^3\frac{\partial}{\partial y}$ \\
   $A_3$  & $3$ & $-4xy^3\frac{\partial}{\partial x}+y^4\frac{\partial}{\partial y}$ \\
   $A_4$   & $4$ & $-5xy^4\frac{\partial}{\partial x}+y^5\frac{\partial}{\partial y}$\\
   $A_5$  & $5$ & $-6xy^5\frac{\partial}{\partial x}+y^6\frac{\partial}{\partial y}$
       \end{tabular}
        \end{center}}
  
For the cases $A_2^{\pm},\ A_4^{\pm}$ we have ommited writing the vector fields for the negative and 
the positive sign since one may be obtained from the other after a multiplication with $-1$ (which means 
that their phase portraits are identical up to a reversal of time).

Except from the hyperbolic model (and the non--vanishing one), they all have an infinity of 
equillibria (the $x$--axis). Othen than that, topologically their behavior is quite simple 
to analyze, since fuction $xf(y)$ serves as a first integral.

It remains to analyze the behavior of pertubations of these vector fields. 
\section{Bifurcations of low codimension}
At regular points, members of $\mathcal{X}(\mathbb{R}^2,xdy)$ are all conjugate 
to each other, via a diffeomorphism preserving the form of Liouville. At hyperbolic singularities 
all such vector fields may be transformed to their linear part; these linear parts are 
not conjugate to each other, since the eigenvalues there are a conjugacy invariant. However, 
up to topological equivalence, they are all saddle points, thus hyperbolic singularities 
are structurally stable.

This is no more the case when we analyze vector fields belonging to the classes $A_k,\ k\geq 1$. 
To describe their local bifurcations we should first compute their transversal unfoldings.

\begin{definition}
Let $X$ be the germ at the origin of a Liouville vector field. Denote by $\mathcal{S}$ its 
singularity class (that is, the set of all germs at the origin of vector fields of Liouville 
which are Liouville equivalent to $X$). A transversal unfolding of $X$ consists of a 
set of germs at the origin of Liouville vector fields, which set intersects $\mathcal{S}$ transversally at $X$.
\end{definition}
Thus, to construct transversal unfoldings of Liouville vector fields, we must first compute 
the tangent spaces of singularity classes.
\begin{theorem}
 Let $X_f\in \mathcal{X}(\mathbb{R}^2,xdy)$ (where $f$ is the function defining $X_f$) 
 and $\mathcal{S}$ its singularity class. We have: 
 \begin{center}
  $T_{X_f}\mathcal{S}=\{X_g \in \mathcal{X}(\mathbb{R}^2,xdy)/ g\in \langle f\rangle \}$.
 \end{center}
\end{theorem}
\begin{proof}
 Let $X_f=-xf'(y)\frac{\partial}{\partial x}+f(y)\frac{\partial}{\partial y}$ be the germ at the origin 
 of a Liouville vector field and $\psi _s(x,y)=(\frac{x}{\varphi _s'(y)},\varphi _s(y))$ the germ at 
 the origin of a family of diffeomorphisms preserving the Liouville form, where $\varphi_0(y)=y$, 
 $\varphi _s(0)=0$ and $\varphi '_s(0)\neq 0$. Define:
 \begin{center}
  $X_s=\psi _{s*}X_f=(-\frac{xf'(y)}{\varphi '_s(y)}-\frac{xf(y)}{(\varphi '_s(y))^2}\varphi ''_s(y))\frac{\partial}{\partial x}+\varphi '_s(y)f(y)\frac{\partial}{\partial y}$.
 \end{center}
It is a curve of Liouville vector fields belonging to $\mathcal{S}$, and we have $X_0=X_f$. To calculate 
the tangent space $T_{X_f}\mathcal{S}$ we need to evaluate at $s=0$ the derivative with respect to 
the parameter $s$ of $X_s$. It is:
\begin{center}
 $\frac{\partial}{\partial s}X_s|_{s=0}=(-xf'(y)\Phi '(y)-x\Phi ''(y)f(y))\frac{\partial}{\partial x}+\Phi '(y)f(y)\frac{\partial}{\partial x}$.
\end{center}
We have denoted as $\Phi (y)$ the vector field defined by $\frac{\partial}{\partial s}\varphi _s(y)=\Phi (\varphi _s(y))$. 
Note that $\frac{\partial}{\partial s}X_s|_{s=0}$ is a Liouville vector field, corresponding to the function 
$\Phi '(y)f(y)$, which belongs to $\langle f\rangle _{\mathcal{E}}$, since $\Phi \in m$. Thus, 
the tangent space of $\mathcal{S}$ at $X_f$ consists of those Liouville fields corresponding to functions 
belonging in the ideal $\langle f\rangle _{\mathcal{E}}$.
\end{proof}

The theorem above allows us to study bifurcations of Liouville vector fields. 
To illustrate this, we present here such bifurcations of low codimension. 

We begin with the singularity class $A_1$. The members of this class form a subset of codimension 
$1$ in the set of those members of $\mathcal{X}(\mathbb{R}^2,xdy)$ vanishing at the origin. To transversally unfold 
them, we only need to add to their local model, linear terms preserving the form of Liouville.
We arrive thus at the vector field $Q_a(x,y)=(-ax-2xy)\frac{\partial}{\partial x}+(ay+y^2)\frac{\partial}{\partial y}$, 
where $a$ a real parameter.
We have the following:
\begin{proposition}
The set of $X\in \mathcal{X}(\mathbb{R}^2,xdy)$ with $j^0X(0)=j^1X(0)=0$ and $j^2X(0)\neq 0$ 
has codimension 1 in the set of Liouville vector fields vanishing at the origin. Its members 
are all conjugate to the $A_1$ model given above. The curve of vector fields $Q_a(x,y)$ intersects 
at $a=0$ this set transversally.
\end{proposition}
\begin{proof}
The codimension and the conjugacy to the $A_1$ model follows easily from the analysis given in 
the previous sections. Note that $Q_0(x,y)$ is the $A_1$ model, corresponding to the function 
$y^2$. The intersection is transversal, since:
\begin{center}
 $\frac{\partial}{\partial a}Q_a(x,y)|_{a=0}=-x\frac{\partial}{\partial x}+y\frac{\partial}{\partial y}$.
\end{center}
This is a Liouville vector field corresponding to the function $y$ which is the only function 
(up to a constant) which vanishes 
at the origin and belongs to $\mathcal{E}/\langle y^2\rangle$.
\end{proof}
Thus, $Q_a(x,y)$ is a transversal unfolding of the $A_1$ singularity. Vector fields depending 
on a single parameter undergoe, for isolated values of this parameter, the bifurcation depicted 
in Figure 1; this bifurcation is therefore the codimension 1 bifurcation occuring in vector fields of 
interest.
\begin{center}
\begin{figure}
\includegraphics[height=5.cm]{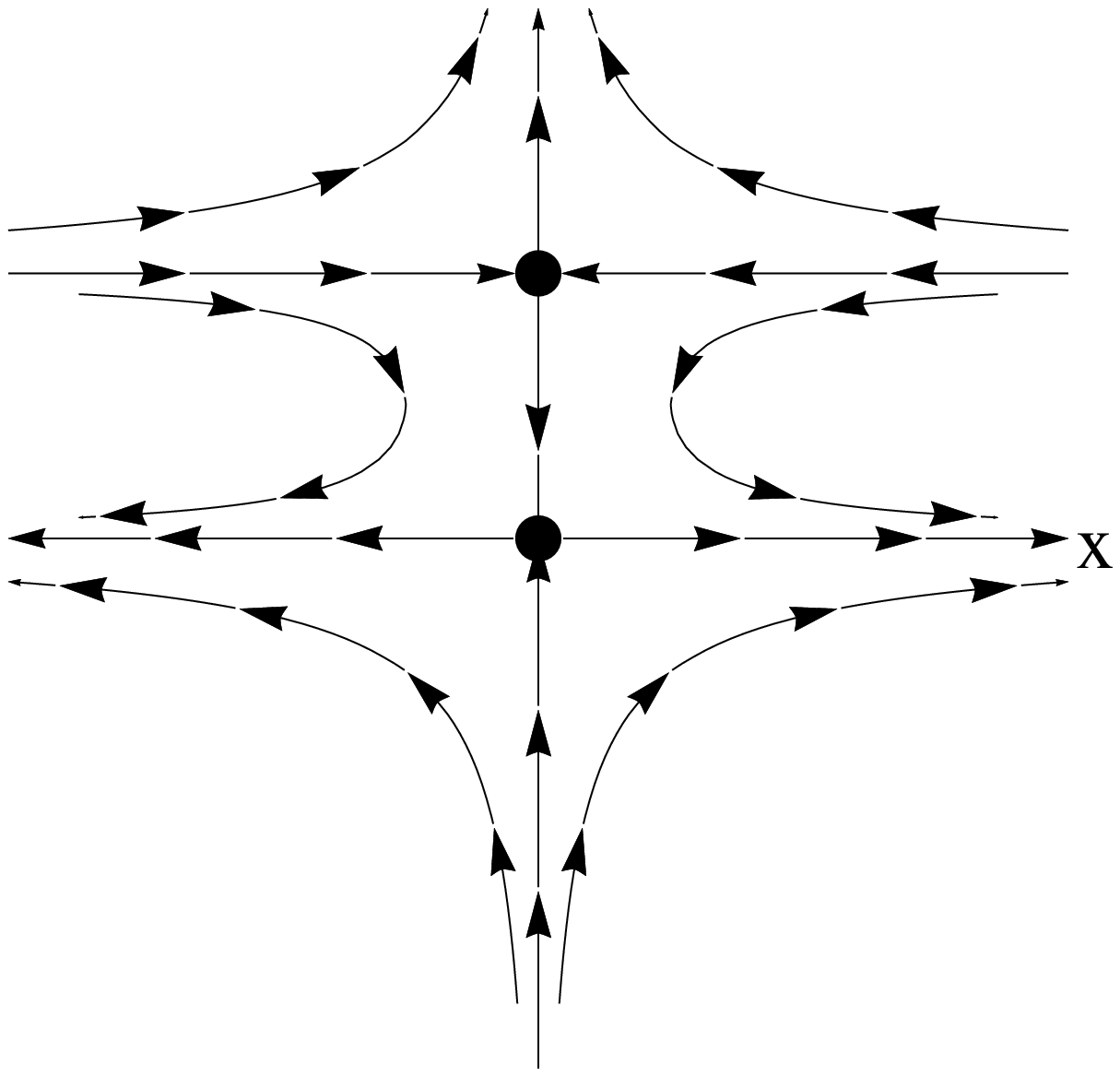}
\includegraphics[height=5.cm]{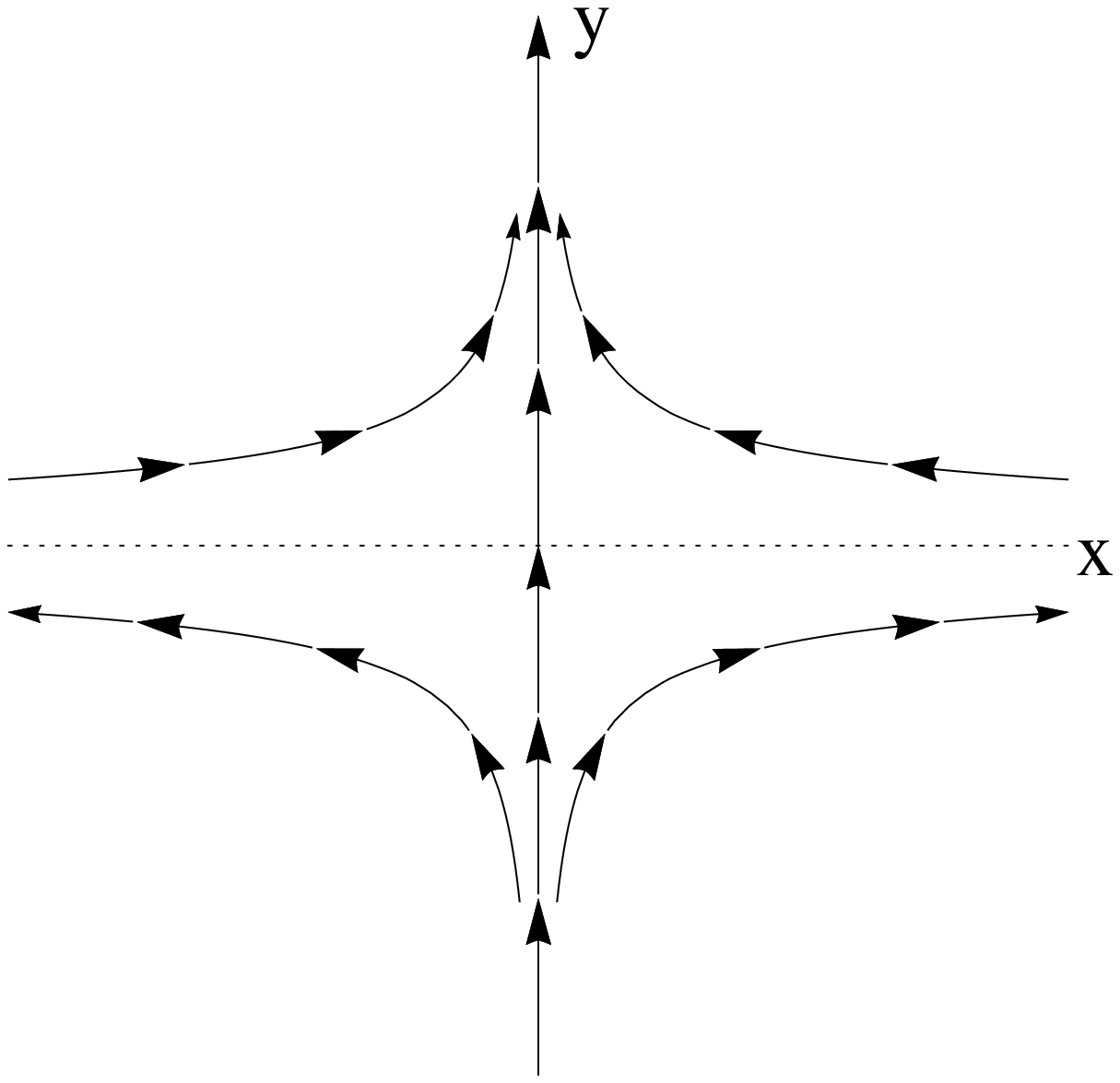}
\includegraphics[height=5.cm]{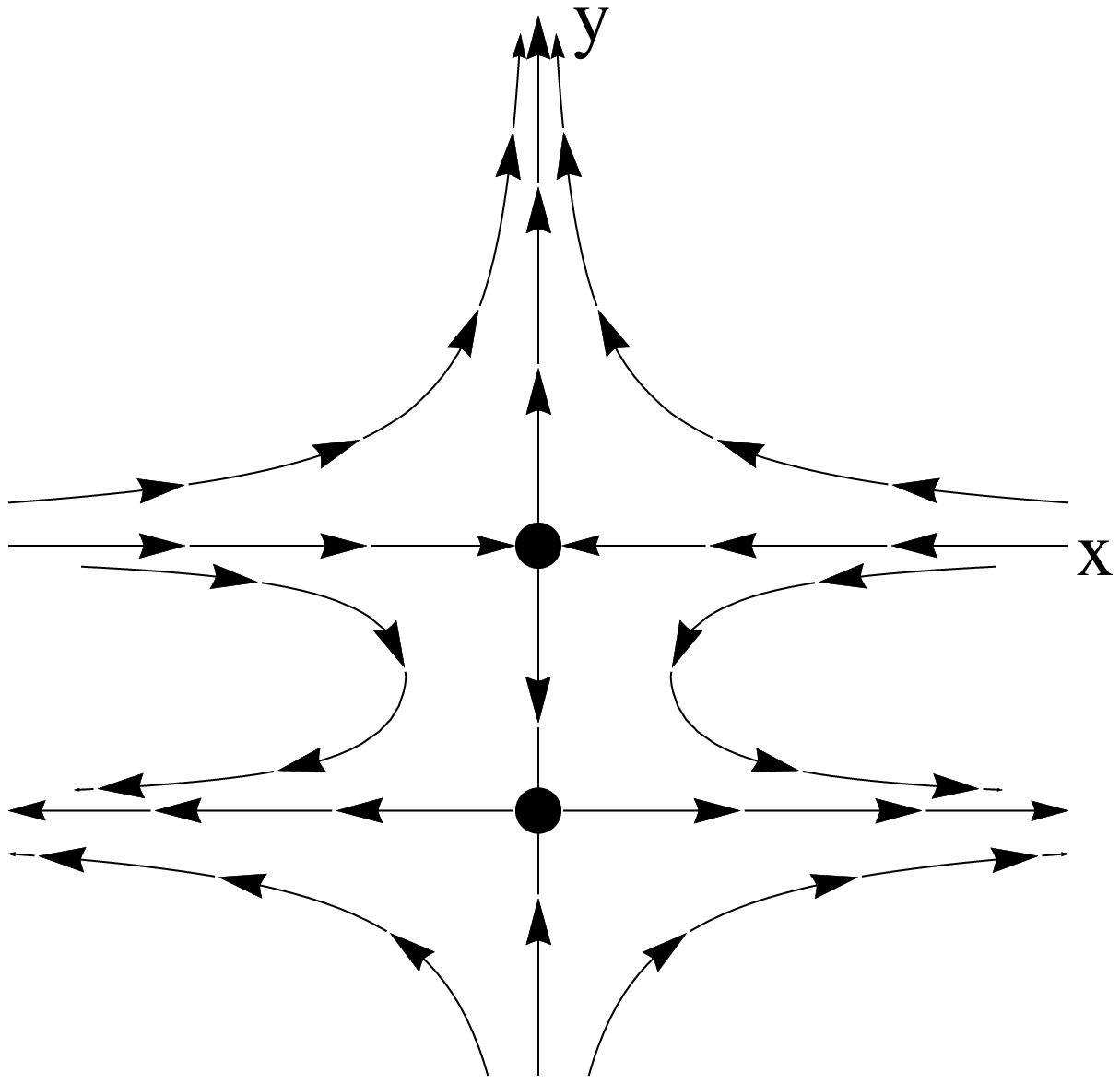}
\caption{Bifurcation of codimension one: $a<0,\ a=0,\ a>0$. The dotted line in the center picture stands 
for the line of singularities.}
\end{figure} 
\end{center} 
We proceed to bifurcations of codimension two. Consider the $A_2$ model, and add to it 
terms of lower degree. We arrive at $T_{a,b}=(-ax-2bxy-3xy^2)\frac{\partial}{\partial x}+(ay+by^2+y^3)\frac{\partial}{\partial y}$, 
where $a,\ b$ real parameters. We have the following:
\begin{proposition}
The set of $X\in \mathcal{X}(\mathbb{R}^2,xdy)$ with $j^0X(0)=j^1X(0)=j^2X(0)=0$ and $j^3X(0)\neq 0$ 
has codimension 2 in the set of Liouville vector fields vanishing at the origin. Its members 
are all conjugate to the $A_2$ model given above. The surface of vector fields $T_{a,b}(x,y)$ intersects 
at $a=b=0$ this set transversally.
\end{proposition}
Its proof goes along the lines of the previous proposition, and it is therefore omitted. 
In figure 2 we present the bifurcations system $T_{a,b}$ system undergoes, for characteristic 
parameter values. 

Before discussing the diffeomorphism case, let us note that we could study bifurcations 
of arbitrary, finite, codimension following the exact same approach.
\begin{center}
\begin{figure}
\includegraphics[height=5.cm]{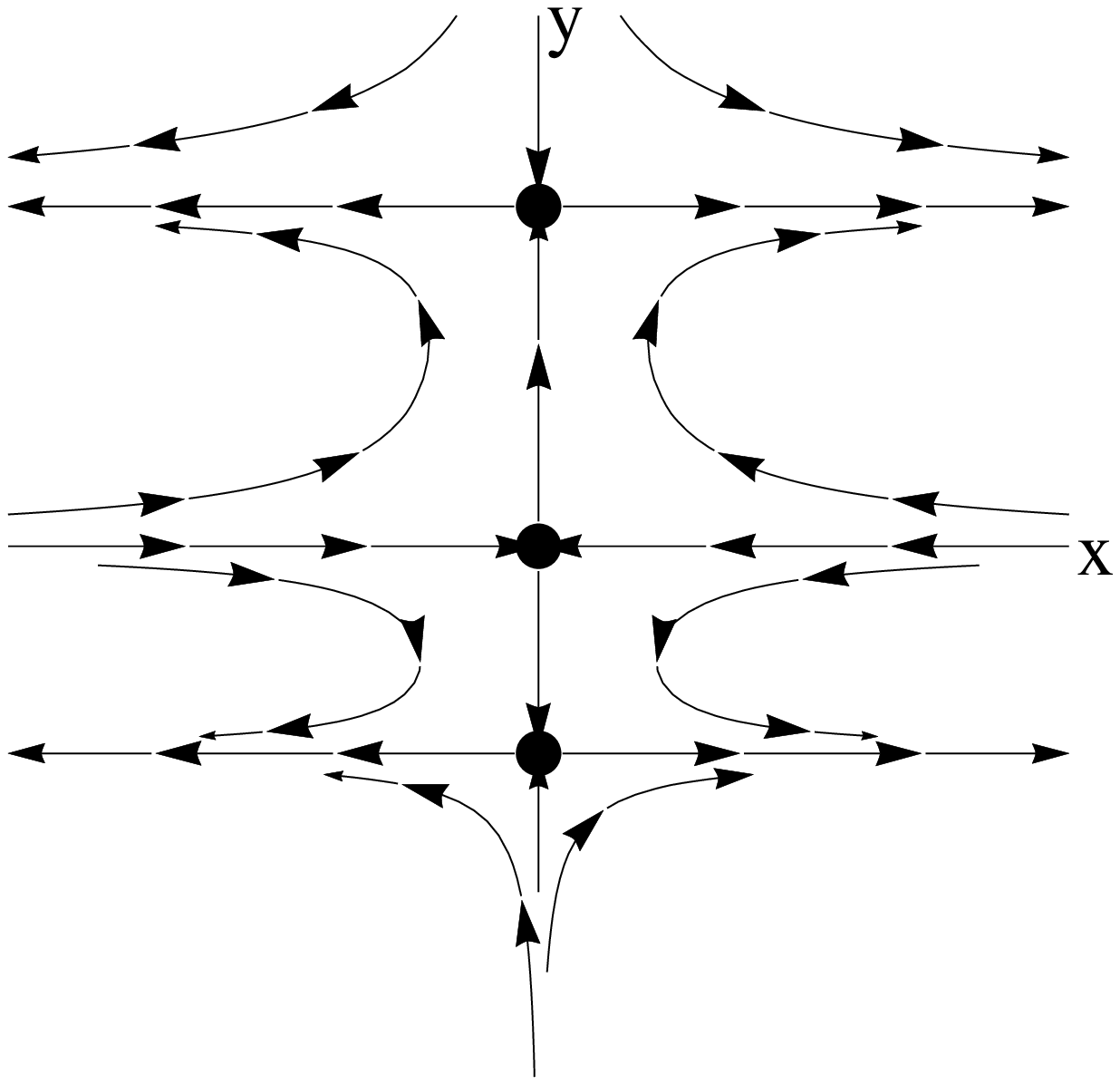}
\includegraphics[height=5.cm]{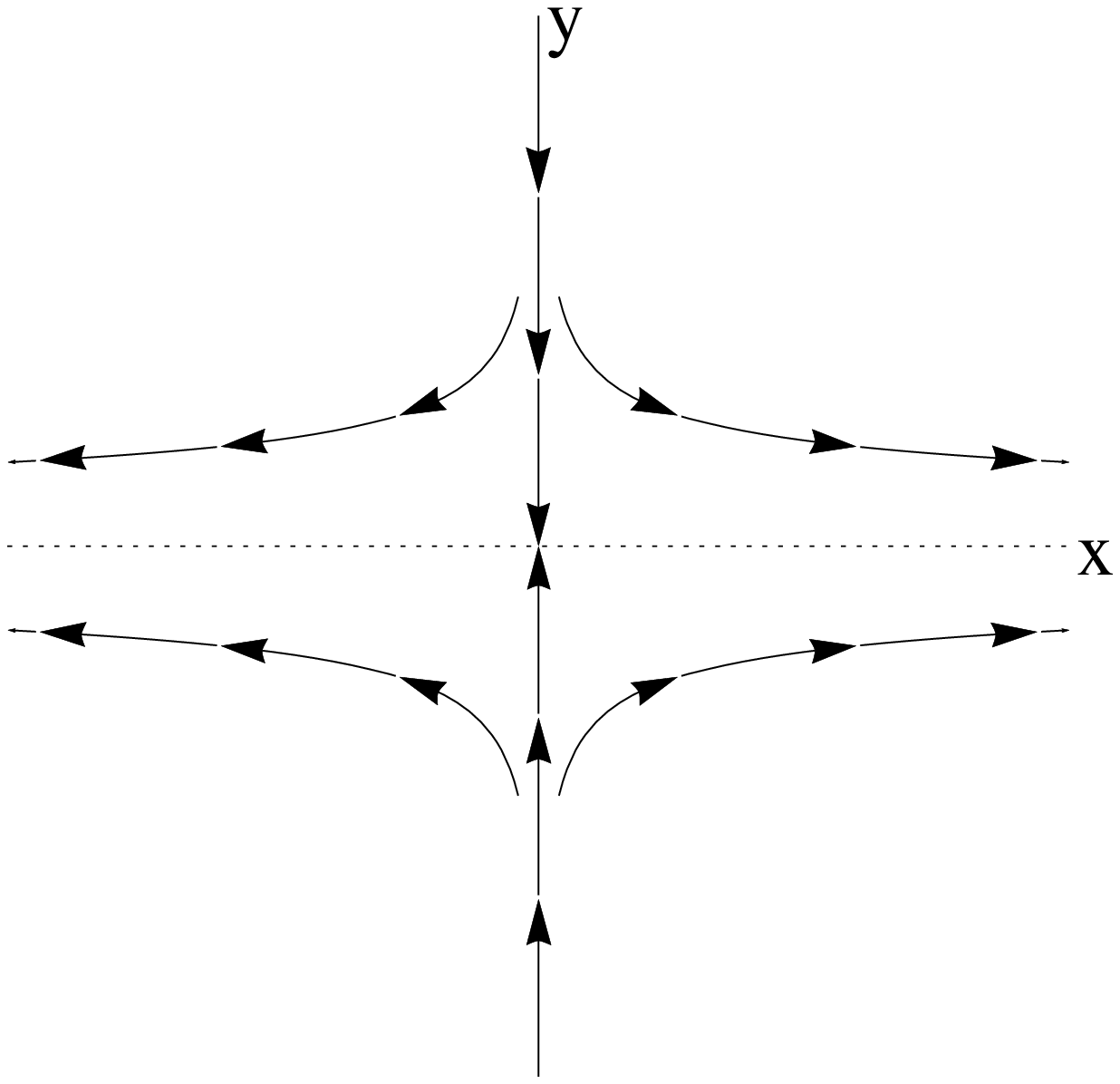}
\includegraphics[height=5.cm]{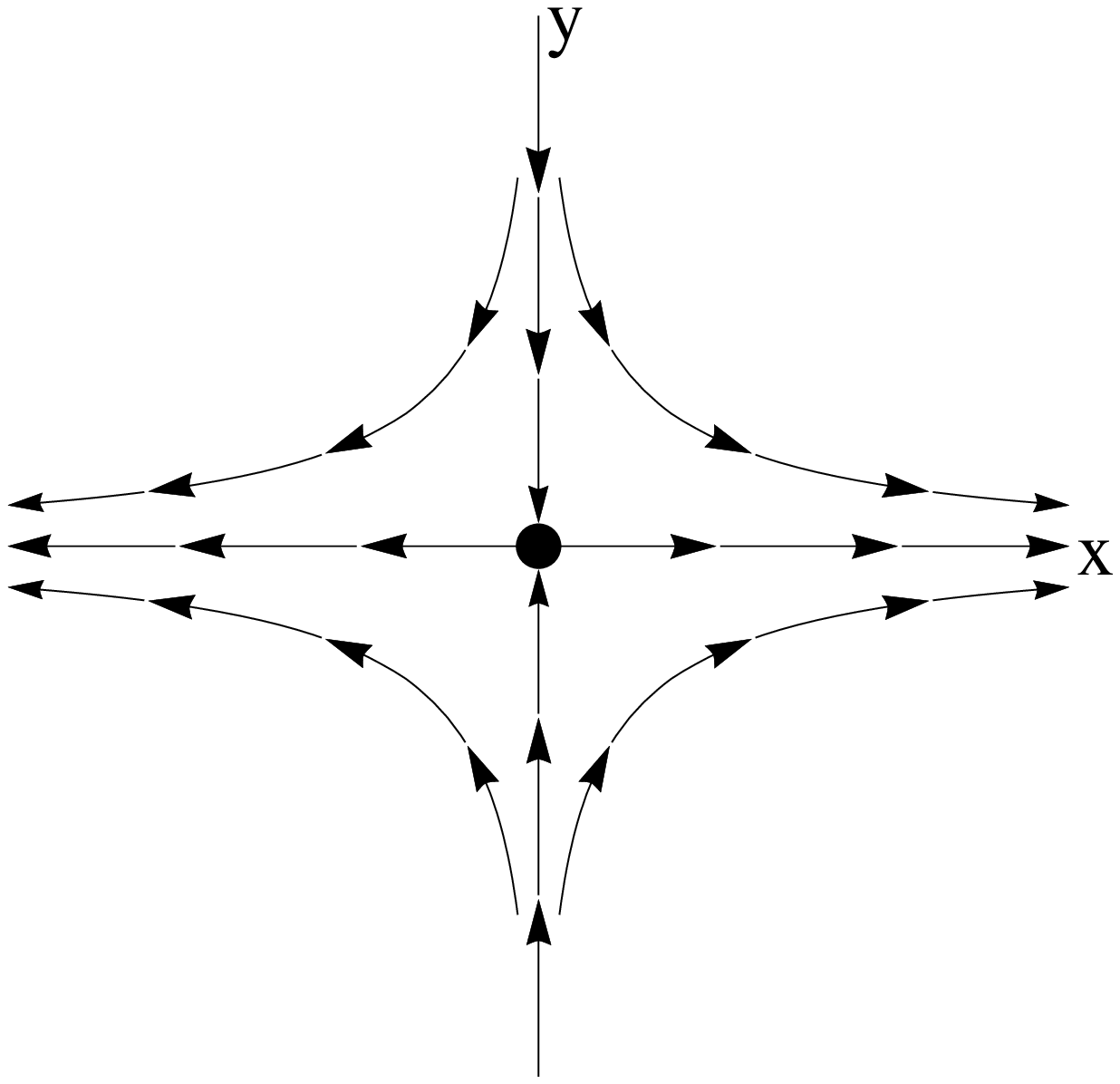}
\caption{Bifurcations of codimension two: $a_n,\ b_n <0$ (left), $a=b=0$ (center), $a>0,\ b=0$ (right). 
The dotted line in the center picture stands for the line of singularities.}
\end{figure} 
\end{center}
\section{Plane diffeomorphisms preserving the form of Liouville}
Let us now turn our attention to diffeomorphisms of the plane respecting the form of Liouville. 
As we saw, they are of the general form $f(x,y)=(\frac{x}{h'(y)},h(y))$. Diffeomorphism 
$h(y)$ of $\mathbb{R}$ uniquely defines such a diffeomorphism.

The unique linear diffeomorphism preserving the form of Liouville (and the origin) is thus 
$(x,y)\mapsto (ax,\frac{1}{a}y)$. Aside from this, there are no other polynomial 
members of $Diff(\mathbb{R}^2,xdy)$; as a consequence, finite jets (of any order) 
of Liouville diffeomorphisms studied here do not belong to the same group.

The classification of strict contactomorphisms, according to the natural 
equivalence relation, is of course our purpose; $f,g\in Diff(\mathbb{R}^2,xdy)$ 
are Liouville conjugate if there exists a third Liouville diffeomorphism 
$\phi$ such that $f\circ \phi=\phi\circ g$. To continue, and since we focus on 
fixed points, we impose the conditions $f(0)=g(0)=0$.

Generically, such diffeomorphisms may be linearized in a neighborhood of the origin.
\begin{proposition}
There exists a codimension zero subset of those members of $Diff(\mathbb{R}^2,xdy)$ vanishing 
at the origin, every member 
of which may be transformed, via a change of coordinates preserving the Liouville form, 
to its linear part.  
\end{proposition}
\begin{proof}
Let us consider the set of Liouville diffeomorphisms having linear part 
$(ax,\frac{1}{a}y),\ a\neq \pm 1$. Its codimension is zero (in the set of Liouville diffeomorphisms vanishing 
at the origin) and its members are of the form $f(x,y)=(\frac{x}{h'(y)},h(y))$ 
where $h(y)=\frac{1}{a}y+h.o.t.$ a local diffeomorphism (we use h.o.t. as 
an abbreviation for ''higher order terms"). 

We have supposed that $a\neq\pm 1$; therefore a local diffeomorphism $\psi$ of $\mathbb{R}$ 
exists such that $\psi\circ h\circ \psi^{-1}=\frac{1}{a}y$ (this is the content of 
the Sternberg linearization theorem, see \cite{sternberg}). Using this 
diffeomorphism define $\phi (x,y)=(\frac{x}{\psi'(y)},\psi(y))$ and 
observe that it is a diffeomorphism, preserving the Liouville form, with 
inverse $\phi^{-1}(x,y)=(\frac{x}{(\psi^{-1}(y))'},\psi^{-1}(y))$.

As is easy to confirm, $\psi\circ f\circ\psi^{-1}=(ax,\frac{1}{a}y)$. 
\end{proof}
We have thus found the generic model for the mappings under study,that is $(x,y)\mapsto (ax,\frac{1}{a}y)$. 
As already remarked, it is actually the unique polynomial model for members of 
$Diff(\mathbb{R}^2,xdy)$; thus Liouville diffeomorphisms either may be linearized 
or are not finitely determined (at least finitely determined under the relation of Liouville 
conjugacy).
\section{Homogeneous members of $\mathcal{X}(\mathbb{R}^3,a)$ and linearization}
Having completed the study of vector fields of Liouville we may now state results 
for strictly contact vector fields of $\mathbb{R}^3$. Indeed, one needs only to add 
constant multiples of $\frac{\partial}{\partial z}$ to the local models presented above, 
to obtain vector fields which preserve both the contact form $a$ and the form of Liouville. 

Our choise of restricting our study to members of $\mathcal{X}_L(\mathbb{R}^3,a)$ stems 
from the fact that they are the only strictly contact vector fields which may have homogeneous 
components. Indeed, recall from section 1 the general form of a strictly contact vector field:
\begin{center}
$X=-\frac{\partial H(x,y,z)}{\partial y}\frac{\partial}{\partial x}+\frac{\partial H(x,y,z)}{\partial x}\frac{\partial}{\partial y}+(H(x,y,z)-x\frac{\partial H(x,y,z)}{\partial x})\frac{\partial}{\partial z}$.
\end{center}
Assuming that $H(x,y,z)$ (remember it does not depend on $z$) is a homogenous polynomial of degree $d$, 
vector field $X$ above is homogeneous of degree $d-1$ only in case its third component is constant, for $d=1$, 
or zero, for $d\geq 2$. Members 
of $\mathcal{X}_L(\mathbb{R}^3,a)$ are therefore the only homogeneous members of $\mathcal{X}(\mathbb{R}^3,a)$. 
We shall elaborate in this observation in this section, to show, using classical normal form 
theory, the linearization of strictly contact vector fields respecting the form of Liouville.

Consider members of $\mathcal{X}(\mathbb{R}^3,a)$ vanishing at the origin. If $X$ is such a field, 
let $X=X_1+X_2+...+X_k$ be its $k$--jet at zero, for some natural number $k$, where each $X_i,\ i=1,..k,$ 
is a homogeneous field of degree $k$. It is easy to see, equating terms of the same degree in 
equation $\mathcal{L}_X(a)=0$, that each $X_i$ is itself a member of $\mathcal{X}(\mathbb{R}^3,a)$.

We denote as $\mathcal{X}^d(\mathbb{R}^3,a)$ the subset of $\mathcal{X}(\mathbb{R}^3,a)$, 
the components of which are homogeneous functions of degree d. We easily prove the following:
\begin{lemma}
The vector space $\mathcal{X}^d(\mathbb{R}^3,a)$ is one dimensional. For each $d\in\mathbb{N}\setminus \{0\}$, 
its base consists of the field $X_d=dxy^{d-1}\frac{\partial}{\partial x}-y^d\frac{\partial}{\partial y}$.
\end{lemma}
The local models of Table 2 constitute, therefore, the basis generating the fields of interest.

Linear fields (belonging to $\mathcal{X}^1(\mathbb{R}^3,a$) are of the form 
$X_1=ax\frac{\partial}{\partial x}-ay\frac{\partial}{\partial y}$, with $a$ 
arbitrary constant. In our case, therefore, the existence of hyperbolic singularities 
is excluded (actually, $X_1$ is also the unique linear member of $\mathcal{X}(\mathbb{R}^3,\alpha)$; 
strictly contact vector fields do not possess hyperbolic singularities). Despite this fact, fields having non--zero linear part can be 
linearized, in a neighborhood of the origin. We shall prove it now 
using an approach different from the one indicated above.

There are $\frac{1}{2}(d^2+3d+2)$ monomials depending on three variables, 
having degree $d$, as simple counting arguments may assure. Thereupon, the 
vector space $\mathcal{X}^d(\mathbb{R}^3)$ of homogeneous vector fields of 
degree $d$ is of dimension $\frac{3}{2}(d^2+3d+2)$, and one may easily verify 
that the fields appearing in Table 3, being $\frac{3}{2}(d^2+3d+2)$ 
independent vector fields of degree $d$, constitute a basis of it.
\begin{remark}
Vector fields of interest here belong to this base (to obtain them, just set 
$m_1=m_3=0$ to the first field of the second class). This base was presented, 
in the general n--dimensional case, in \cite{meiss}, section 4 of which contains 
the arguments we shall use to prove the next proposition. The author wishes 
to thank Prof. J D Meiss for clarifying them to him.
\end{remark}
\newpage
\begin{center}
Table 3
\end{center}    
{\small \begin{center}
   \begin{tabular}{lll}
   \hline
   fields & condition & number \\
\hline 
\\
$y^{m_1}z^{m_2}\frac{\partial}{\partial x}$

\\

$x^{m_1}z^{m_2}\frac{\partial}{\partial y}$

\\

$x^{m_1}y^{m_2}\frac{\partial}{\partial z}$  &  $m_1+m_2=d$ & $3d+3$  \\
\hline
\\
\vspace{0.2cm}
$(1+m_2)x^{m_1+1}y^{m_2}z^{m_3}\frac{\partial}{\partial x}-(1+m_1)x^{m_1}y^{m_2+1}z^{m_3}\frac{\partial}{\partial y}$ & $\ $&$\ $\\
$(1+m_3)x^{m_1}y^{m_2+1}z^{m_3}\frac{\partial}{\partial y}-(1+m_2)x^{m_1}y^{m_2}z^{m_3+1}\frac{\partial}{\partial z}$ 
& $m_1+m_2+m_3+1=d$ & $d^2+d$\\
\hline
\\
$x^{m_1+1}y^{m_2}z^{m_3}\frac{\partial}{\partial x}+x^{m_1}y^{m_2+1}z^{m_3}\frac{\partial}{\partial y}+x^{m_1}y^{m_2}z^{m_3+1}\frac{\partial}{\partial z}$& $m_1+m_2+m_3+1=d$ & $\frac{1}{2}(d^2+d)$\\
\hline 
       \end{tabular}
        \end{center}}
\vspace{1.cm}
If $X\in\mathcal{X}^d(\mathbb{R}^3)$, the vector field $[X_1,X]$, where $X_1$ is 
the unique, linear and non--zero, strictly contact vector field presented above, 
is also homogeneous of degree d (the brackets $[\cdot ,\cdot ]$ denote the usual 
commutator of vector fields). We may define therefore the operator 
$ad_{X_1}:\mathcal{X}^d(\mathbb{R}^3)\rightarrow \mathcal{X}^d(\mathbb{R}^3)$, 
$X\mapsto [X_1,X]$. Vector fields belonging to the base of $\mathcal{X}^d(\mathbb{R}^3)$ 
are eigenvectors of this operator; thus the subspaces generated by them are 
invariant under $ad_{X_1}$, ensuring the diagonal form of its matrix.
\begin{proposition}
There exists a codimension zero subset of $\mathcal{X}_L(\mathbb{R}^3,\alpha)$ 
every member of which may be transformed to its linear part. The linearizing diffeomorphism is 
close to the identity and preserves the contact form. 
\end{proposition}
\begin{proof}
The subset we refer to is the set of vector fields of interest with non zero 
linearization, and its codimension is easily obtained. 

Classical normal form theory ensures that, by changing coordinates, we may discard all 
terms of $X=X_1+X_2+...\in \mathcal{X}_L(\mathbb{R}^3,\alpha)$ which are not contained in the 
complement of the range of this operator (an operator which leaves invariant the spaces 
$\mathcal{X}^d(\mathbb{R}^3,\alpha)$, as well as the subspaces generated by the basic vector fields, 
the subspace of fields which interest us included). 

The matrix of $X_1$ is self--adjoint, so a complement to the range of $ad_{X_1}$ is the kernel 
of this operator. This kernel however, as may easily be verified, is trivial, providing us with 
a diffeomorphism which transforms to its (non--zero) linear part every field 
of $\mathcal{X}_L(\mathbb{R}^3,\alpha)$. This diffeomorphism preserves the 1--form defining 
the contact structure; this stems from the diagonal form of the matrix of $ad_{X_1}$.
\end{proof}
Strictly contact vector fields project to symplectic fields of the plane; homogeneous 
strictly contact vector fields project to fields of the plane preserving the form of 
Liouville. We have studied here the local behavior of the later; the local study of 
the first remains a challenging task.
\section{Conclusions}
Contact systems have a long history, and attract a lot of attention, since 
they form a valuable tool in topological constructions, in Hamiltonian 
dynamics and in many physical applications (see \cite{geiges} for a textbook 
account of these fields, and further references).

Almost all contact systems possess hyperbolic singularities, as transversality arguments 
show. In this case, conditions for linearization have been obtained (\cite{chaperon2, llave, lychagin}). 
Results are much more rare, however, if the singularities are degenerate. 

We chose here to consider the simpler case of homogeneous strictly contact systems. 
This led us to the study of plane systems, preserving the form of Liouville, a 
subject which has an interest of its own. To study these fields we had to classify 
univariate functions according to the restricted contact equivalence relation. All these 
admit generalizations and deserve more study.

Indeed, extending the definition of restricted contact equivalence to arbitrary dimensions 
we get of course the differential conjugacy relation for vector fields. One could probably 
reobtain results of normal form theory, using this approach, which would potentially help 
the problem of classifying vector fiels preserving the form of Liouville in any dimension. 

And, as already mentioned, the general problem of analyzing the behavior of contact dynamical 
systems stands, both interesting and difficult. The author hopes to further comment on these 
subjects in the future.      
\section*{Acknowledgments} \noindent
This work is dedicated to my two professors, Tassos Bountis and Spyros Pnevmatikos, 
on the occasion of their 65th birthday. It is only a pleasure for the author to acknowledge 
the influence they had on him and to thank them for their constant support.

\begin{thebibliography}{50}

\bibitem{sternberg}
Sternberg S, ''Local $C^n$ transformations of the real line" Duke Math. J., 
24, 97-102, 1957. 

\bibitem{boothby}
Boothby W M, Wang H C, ''On contact manifolds", 
Ann. of Math., 2, 68, 721-734, 1958.

\bibitem{germs}
Br\"{o}cker T, ''Differential Germs and Catastrophes", Cambridge
University Press, 1975.



\bibitem{lychagin}
Ly$\check{c}$agin V V, ''On sufficient orbits of a group of contact diffeomorphisms", 
Math. USSR Sbornik, 33, 2, 223-242, 1977.

\bibitem{damon}
Damon J, ``The unfolding and determinacy theorems for subgroup of $\mathcal{A}$ and 
$\mathcal{K}$'', Mem.of Am.Math.Soc., 50, 306, 1984.

\bibitem{chaperon}
Chaperon Mark, ''G\'eom\'etrie Diff\'erentielle et Singularit\'es des Syst\`emes Dynamiques", 
Ast\'erisque, 138-139, 1986.

\bibitem{arnold1}
Arnol'd V I, Gusein--Zade S M, Varchenko A N, ''Singularities of Differentiable Maps``, Birkhauser, 1986. 

\bibitem{arnold2}
Arnol'd V I, ''Mathematical Methods of Classical Mechanics``, Springer, 1989. 

\bibitem{llave}
Banyaga A, de la Llave R, Wayne C E, ''Cohomology equations near hyperbolic points and 
geometric versions of Sternberg linearization theorem", Journal of Geometric Analysis, 
6(4), 613-649, 1996.
        
\bibitem{haller}
Haller G, Mezic I, ''Reduction of three--dimensional, volume--preserving flows 
with symmetry", Nonlinearity, 11, 319-339, 1998.

     
\bibitem{chaperon2}
Chaperon Mark, ''Singularities in contact geometry", Geometry and Topology of Caustics, 
Caustics '02, Banach Center Publications, Volume 62, Warsaw, 3955, 2004.

\bibitem{geiges}
Geiges H, ''An Introduction to Contact Topology", Cambridge Studies in Advanced Mathematics, vol. 109, 
Cambridge University Press, 2008.

\bibitem{meiss}
Meiss J D, Dullin D R, ''Nilpotent normal form for divergence--free vector 
fields and volume--preserving maps", Phys.D, 237, 156-166, 2008.

\bibitem{takahashi}
Takahashi M, ''A sufficient condition that contact equivalence implies right equivalence 
for smooth function germs``, Houston J.Math., 35, 3, 829-833, 2009.

\bibitem{ninakos}
Kourliouros K, ''Singularities of functions on the Martinet plane, constrained Hamiltonian systems 
and singular lagrangians``, J.Dyn.Control Systems, 21, 3, 401-422, 2015.

  \end{thebibliography}

\end{document}